\documentclass[11pt]{amsart} 
\usepackage{amscd,amssymb}
\usepackage{enumerate} 
\usepackage[all]{xypic} 
\usepackage{multirow}
\usepackage{hhline}
\usepackage{verbatim} 
\usepackage{mathdots} 
\usepackage{comment} 
\usepackage{tikz}
\usepackage{appendix}
\usepackage{stmaryrd}

\usetikzlibrary{decorations.markings}

\theoremstyle{plain} 

\newtheorem*{Thm*}{Theorem} 

\newtheorem{Prop}[equation]{Proposition}
\newtheorem{Lem}[equation]{Lemma}

\theoremstyle{definition}

\newtheorem{Ex}{Example} 

\numberwithin{equation}{section}

\newcommand{\C}{\mathbb C}

\newcommand{\Z}{\mathbb{Z}}
\newcommand{\R}{\mathbb{R}}

\newcommand{\la}{\lambda}
\newcommand{\ep}{\varepsilon}

\newcommand{\Ind}{\operatorname{Ind}}

\newcommand{\sign}{\operatorname{sign}}
\newcommand{\Irr}{\operatorname{Irr}}

\newcommand{\ie}{\text{i.e.\ }}

\newcommand{\GL}{\operatorname{GL}}

\title[Archimedean local converse theorem]{A Local converse theorem
  for Archimedean $\GL(n)$}

\author{Moshe Adrian}
\address{Department of Mathematics 
Queens College, CUNY
65-30 Kissena Blvd., Queens, NY 11367-15971}
\email{moshe.adrian@qc.cuny.edu}

\author{Shuichiro Takeda}
\address{Mathematics Department, University of Missouri-Columbia, 202
  Math Sciences Building, Columbia, MO, 65211}
\email{takedas@missouri.edu} 

\begin{document}

\maketitle

\begin{abstract} 
In this paper, we will prove a version of local
converse theorem for $\GL_n$ over the archimedean local fields which
characterizes an infinitesimal equivalence class of irreducible
admissible representations of $\GL_n(\R)$ (or $\GL_n(\C)$)
in terms of twisted local $L$-factors. 
\end{abstract}

%%%%%%%%%%%%%%%%%%%%%%%%%%%%%%%%%%%%%%%%%%%%%%%%%%%%%%%%%%%%%%%%%%%

\section{Introduction}

%%%%%%%%%%%%%%%%%%%%%%%%%%%%%%%%%%%%%%%%%%%%%%%%%%%%%%%%%%%%%%%%%%%

Let $F$ be a local field of characteristic 0, and let $\Irr_n$ be the set
of (infinitesimal) equivalence classes of irreducible admissible
representations of $\GL_n(F)$. A so-called local converse theorem for
$\GL_n(F)$ characterizes the set $\Irr_n$ in terms of local factors
with some suitable twists. If $F$ is
non-archimedean, the first major result is the one by Henniart
(\cite{Henniart1}) in which he shows that if two {\it generic}
representations $\pi, \pi'\in\Irr_n$ are such that
\[
\gamma(s, \pi\otimes\tau, \psi)=\gamma(s, \pi'\otimes\tau, \psi)
\]
for all $\tau\in\Irr_t$ for all $t=1,\dots,n-1$, where the
$\gamma$-factor is the one defined by Jacquet, Piatetski-Shapiro and
Shalika, then
$\pi=\pi'$. Later, in \cite{Chen} Chen improved this result by requiring
$t$ be only up to $n-2$ with the extra assumption that $\pi$ and
$\pi'$ have the same central character (we note that it was later known that $\GL_1$ twisting determines the central character; see, for example, \cite{JNS}). It had been conjectured by
Jacquet for some time that one only needs $t\leq[\frac{n}{2}]$. And very
recently this conjecture has been proven by Chai in \cite{Chai}, and Jacquet and Liu
in \cite{JL} (see also \cite{ALSX, JNS}). Let us also mention that Nien in \cite{Nien} has shown
an analogous result when $F$ is a finite field.

In this paper, we will prove a local converse theorem when $F$ is
archimedean by using $L$-factors of Artin type (without the
assumption of the genericity and the central character) with only up
to $\GL(1)$-twists for the complex case and up to $\GL(2)$-twists for the
real case. Namely, we will prove
\begin{Thm*}[Complex Case]
Let $F=\C$. If the two representations  $\pi, \pi'\in\Irr_n$ of
$\GL_n(\C)$ satisfy
\[
L(s, \pi\times\chi)=L(s, \pi'\times\chi)
\]
for all characters $\chi$ on $\C^\times$, where the $L$-factors are
of Artin type defined by the local Langlands correspondence, then $\pi=\pi'$. 
\end{Thm*}

\begin{Thm*}[Real Case]
Let $F=\R$. If the two representations $\pi, \pi'\in\Irr_n$ of $\GL_n(\R)$ satisfy
\[
L(s, \pi\times\tau)=L(s, \pi'\times\tau)
\]
for all $\tau\in\Irr_t$ with $t=1, 2$, where the $L$-factors are
of Artin type defined by the local Langlands correspondence, then $\pi=\pi'$. 
\end{Thm*}

We will show that for $F=\R$ the bound $t\leq 2$ is sharp even
if $n=2$, namely one always needs $\GL(2)$-twists. We will show,
however, that if we assume that $\pi$ and $\pi'$ have the same
central character and $n\leq 3$, then one only needs up to
$t=1$. However, even with the central character assumption, for
$n\geq 4$ one always needs $\GL(2)$-twists, and indeed we will give an
infinite family of nonequivalent representations of $\GL_4(\R)$ which cannot be
distinguished by $\GL(1)$-twists.

Let us note that the local converse theorems via gamma factors as in
\cite{Henniart1, Chen, Chai, JL} all assume that the representations
are generic. This is because gamma factors cannot distinguish the
representations appearing in a same parabolically
induced representation. Namely, all the constituents of a parabolically
induced representation have the same gamma factor, and accordingly the
genericity assumption is necessary. Our theorem should rather be
considered as an archimedean analogue of the local converse theorem by
Henniart in \cite{Henniart2} in which he characterizes the local
Langlands correspondence for $p$-adic $\GL_n$ via $L$- and
$\epsilon$-factors without the genericity assumption. Of course, it
makes sense to ask if it is possible to establish a local converse
theorem for the archimedean case in terms of gamma factors, for generic
representations. We will take up this issue in our later work.

The basic idea of our proof is that we pass to the ``Galois side'' via
the local Langlands correspondence, so that
the local $L$-factors, which are then essentially products of gamma
functions, can be explicitly computed in terms of the data for the
corresponding representations of the Weil group, and then we will
compare poles of the gamma functions.

The structure of the paper is as follows: In the next section (Section
2), we will review basics of the local Langlands parameters and their
local factors. In Section 3, we will
introduce a certain partial order on complex numbers, which will be
useful when we compare poles of local $L$-factors. In Sections 4 and 5,
we will prove our theorems for the complex and real case,
respectively. Finally in Section 6, we will discuss some of the low
ranks cases.

\quad

\begin{center}{\bf Notations}\end{center}

Throughout, $F$ is either $\R$ or $\C$. We let $\Irr_n$ be the
set of infinitesimal equivalence classes of irreducible admissible
representations of $\GL_n(F)$. For $z\in F$, we let
$|z|=\sqrt{z\bar{z}}$, so that if $F=\R$, this is the absolute value of
$z$, and if $F=\C$, it is the usual modulus of $z$. We also let
$\|z\|=z\bar{z}=|z|^2$. By a character, we always mean a
quasi-character, and $\Irr_1$ is
the set of characters of $F^\times$. We let $\psi_F$ be the standard
choice of additive character on $F$; namely if $F=\R$, then
$\psi_\R(r)=e^{2\pi i r}$, and if $F=\C$, then
$\psi_\C(z)=\psi_\R\circ\operatorname{Tr}_{\C/\R}(z)=e^{2\pi
  i(z+\bar{z})}$. We let $\Gamma(s)$ be the gamma function. Recall
that $\Gamma(s)$ satisfies
\[
\Gamma\left(\frac{s}{2}\right)\Gamma\left(\frac{s+1}{2}\right)
=2^{1-s}\sqrt{\pi}\Gamma(s),\quad\text{(duplication formula)}.
\]
Also recall that $\Gamma(s)$ has no zeroes, and has infinitely many
poles, which are precisely at $s=0, -1,
-2, \dots$ and all of which are simple.

\quad

\begin{center}{\bf Acknowledgements}\end{center}

This project was initiated when both authors attended the conference
``The 2016 Paul
J.\ Sally,\ Jr.\ Midwest Representation Theory Conference: In honor of
the 70th Birthday of Philip Kutzko'' at the University of Iowa, which
was funded by the University of Iowa and the NSF. We
would like to thank them for their support and also would like to thank the
conference organizers for providing us with the good opportunity for
the interaction. 

The first named author was supported by a grant from the Simons
Foundation \#422638 and by a PSC-CUNY award, jointly funded by the
Professional Staff Congress and The City University of New York.

The second named author was supported by Simons Foundations
Collaboration Grant \#35952, and NSA Young Investigator Grant
H98230-16-1-0312.

\quad

%%%%%%%%%%%%%%%%%%%%%%%%%%%%%%%%%%%%%%%%%%%%%%%%%%%%%%%%%%%%%%%%%%%

\section{Local Langlands parameters for $\GL_n$}

%%%%%%%%%%%%%%%%%%%%%%%%%%%%%%%%%%%%%%%%%%%%%%%%%%%%%%%%%%%%%%%%%%%

In this section, we recall the basics of the local Langland parameters
for $\GL_n(F)$, \ie the $n$-dimensional continuous complex representations of
the Weil group of $F$, and their local factors. The most definitive reference
is \cite{Tate}.

\subsection{Weil group}
We let $W_F$ be the Weil group of $F$. So $W_\C=\C^\times$ and
$W_\R=\C^\times\cup j\C^\times$ with $j^2=-1$ and
$jzj^{-1}=\bar{z}$. We naturally view $W_\C=\C^\times$ as a subgroup of
$W_\R$. Note that $F^\times\cong W_F^{ab}$. This is obvious if
$F=\C$. If $F=\R$, we have a surjective map
\begin{equation}\label{E:W_R}
W_\R\longrightarrow \R^\times, \quad z\mapsto z\bar{z},\; j\mapsto -1,
\end{equation}
whose kernel is the commutator group $[W_\R, W_\R]$, which is actually of the form
$\{z\in\C^\times:|z|=1\}$.

\subsection{Characters on $F^\times$}
We will describe all the characters of $F^\times$. If $F=\C$,
each character is of the form 
\[
\chi_{-N, t}(z):=z^{-N}\|z\|^t,
\]
where $N\in\Z$ and $t\in\C$.  Let us note that if we write $z=re^{i\theta}$ with $r,
\theta\in\R$ as usual, we have
\[
\chi_{-N,t}(z)=r^{2t-N}e^{-iN\theta}.
\]
But when dealing with the local factors, it seems to be more
convenient to denote each character as $z^{-N}\|z\|^t$ instead of
using $re^{i\theta}$, and hence we
will choose this convention.
Let us note that
\[
\overline{\chi_{-N,t}}=\chi_{N, t-N},
\]
where $\overline{\chi_{-N,t}}(z):=\overline{\chi_{-N,t}(z)}=\chi_{-N,t}(\bar{z})$ as usual.

If $F=\R$, each character is of the form
\[
\la_{\ep, t}(r):=r^{-\ep}|r|^t=\sign(r)^{\ep}|r|^{t-\ep},\quad r\in\R^\times,
\]
where $\ep\in\{0,1\}$, $t\in\C$ and $\sign$ is the sign character.

\subsection{Representations of $W_F$}
If $F=\C$, an irreducible representation
of $W_\C=\C^\times$ is 1-dimensional, namely a character on
$\C^\times$, and hence of the form $\chi_{-N,t}$ as above. Then in
general, an $n$-dimensional representation
$\varphi:W_\C\rightarrow\GL_n(\C)$ is of the form
\begin{equation}\label{E:general_parameter_C}
\varphi=\chi_{-N_1,t_1}\oplus\cdots\oplus\chi_{-N_n, t_n}.
\end{equation}

If $F=\R$, an irreducible representation of $W_\R$ is 1 or 2 dimensional. If it
is 1-dimensional, it factors through $W_\R^{ab}\cong\R^\times$ and
hence is identified with a character of the form $\la_{\ep, t}$. If it
is 2-dimensional, it is of the form
\[
\varphi_{-N,t}:=\Ind_{W_\C}^{W_\R}\chi_{-N,t}.
\]
From the definition of $W_\R$, we can see that
$\det(\varphi_{-N,t})(z)=\chi_{-N, t}(z\bar{z})=\|z\|^{-N+2t}$ for
$z\in\C^\times\subseteq W_\R$, and
$\det(\varphi_{-N,t})(j)=-\chi_{-N,t}(-1)=-(-1)^N$. Hence from
\eqref{E:W_R}, as a character on $\R^\times$ we can identify $\det(\varphi_{-N,t})$ with
$\sign\chi_{-N, t}$ where $\chi_{-N, t}$ is viewed as a character on $\R^\times$ via the
inclusion $\R^\times\subseteq\C^\times$. Namely
\begin{equation}\label{E:central_character}
r\mapsto\sign(r) r^{-N}|r|^{2t}.
\end{equation}

Let us mention that if $N=0$, the representation $\varphi_{-N,t}$ is not irreducible but
we have
\[
\varphi_{0, t}=\la_{0, t}\oplus\la_{1, t+1}.
\]
Also since
$\Ind_{W_\C}^{W_\R}\chi_{-N,t}=\Ind_{W_\C}^{W_\R}\overline{\chi_{-N,t}}$,
we have 
\begin{equation}\label{E:N>0}
\varphi_{-N,t}=\varphi_{N, t-N}.
\end{equation}
Hence we may assume
$N>0$. Namely the irreducible 2-dimensional representations of $W_\R$
are precisely the representations of the form $\varphi_{-N,t}$ with
$N>0$. In
general, an $n$-dimensional representation
$\varphi:W_\R\rightarrow\GL_n(\C)$ is of the form
\begin{equation}\label{E:general_parameter_R}
\varphi=\left(\la_{\ep_1,t_1}\oplus\cdots\oplus\la_{\ep_p,t_p}\right)\oplus
\left(\varphi_{-N_1,u_1}\oplus\cdots\oplus\varphi_{-N_q, u_q}\right)
\end{equation}
where $N_i>0$ for all $i$, and $p+2q=n$.

\subsection{$L$- and $\epsilon$-factors}
We will recall the $L$- and $\epsilon$-factors of those
representations. (Though we will not use the $\epsilon$-factors in our
proofs, we will include them for completeness.)
Assume $F=\C$. Then the $L$- and $\epsilon$-factors of the character
$\chi_{-N,t}$ are defined as
\begin{align}
\notag L(\chi_{-N, t})&=2(2\pi)^{-(t-\frac{N}{2}+\frac{|N|}{2})}\Gamma(t-\frac{N}{2}+\frac{|N|}{2})\\
\label{E:L-factor-C}&=\begin{cases}2(2\pi)^{-t}\Gamma(t),\quad\text{if
    $N\geq 0$};\\
2(2\pi)^{-(t-N)}\Gamma(t-N),\quad\text{if $N<0$},
\end{cases}\\
\label{E:ep-factor-C}\epsilon(\chi_{-N,t},\psi_\C)&=i^{|N|}.
\end{align}
Let us note that
\begin{equation*}
L(\chi_{-N, t})=L(\overline{\chi_{-N, t}})\quad\text{and}\quad
\epsilon(\chi_{-N,t},\psi_\C)=\epsilon(\overline{\chi_{-N,t}},\psi_\C).
\end{equation*}
In general, if $\varphi:W_\C\rightarrow\GL_n(\C)$ is an
$n$-dimensional representation as in \eqref{E:general_parameter_C}, we
define the local factors multiplicatively as
\begin{align*}
L(\varphi)&=L(\chi_{-N_1,t_1})\cdots L(\chi_{-N_n, t_n}),\\
\epsilon(\varphi,\psi_\C)
&=\epsilon(\chi_{-N_1,t_1}, \psi_\C)\cdots\epsilon(\chi_{-N_n, t_n}, \psi_\C).
\end{align*}

Assume $F=\R$. For the 1-dimensional $\la_{\ep, t}$, 
\begin{align}
\label{E:L-factor-R1}L(\la_{\ep,t})&=\pi^{-\frac{t}{2}}\Gamma\left(\frac{t}{2}\right),\\
\label{E:ep-factor-R1}\epsilon(\la_{\ep,t},\psi_\R)&=(-i)^\ep.
\end{align}
For the 2-dimensional representation $\varphi_{-N,t}$,
\begin{align}
\label{E:L-factor-R2}L(\varphi_{-N,t})&=L(\chi_{-N,t})
=\begin{cases}2(2\pi)^{-t}\Gamma(t),\quad\text{if $N\geq 0$};\\
2(2\pi)^{-(t-N)}\Gamma(t-N),\quad\text{if $N<0$},
\end{cases}\\
\label{E:ep-factor-R2}\epsilon(\varphi_{-N,t},\psi_\R)&=-i\cdot\epsilon(\chi_{-N,t},\psi_\C)=-i^{|N|+1}.
\end{align}
In general, if $\varphi:W_\R\rightarrow\GL_n(\C)$ is an
$n$-dimensional representation as in \eqref{E:general_parameter_R}, we
again define the local factors multiplicatively as
\begin{align*}
L(\varphi)&=L(\la_{\ep_1,t_1})\cdots L(\la_{\ep_p,t_p})\cdot 
L(\varphi_{-N_1,u_1})\cdots L(\chi_{-N_q, u_q})\\
\epsilon(\varphi,\psi_\R)
&=\epsilon(\la_{\ep_1,t_1}, \psi_\R)\cdots\epsilon(\la_{\ep_p,t_p}, \psi_\R)\cdot 
\epsilon(\varphi_{-N_1,u_1}, \psi_\R)\cdots \epsilon(\chi_{-N_q, u_q},\psi_\R).
\end{align*}
Let us note that for the parameter $\varphi_{0, t}$, we can check that $L(\varphi_{0,
  t})=L(\chi_{0, t})=2(2\pi)^{-t}\Gamma(t)$, which is indeed equal to
$L(\la_{0,t})L(\la_{1,t+1})=\pi^{-t/2}\Gamma(\frac{t}{2})\cdot\pi^{-(t+1)/2}\Gamma(\frac{t+1}{2})$
by the duplication formula. Also
one can check that
$\epsilon(\varphi_{0,t})=-i\cdot\epsilon(\chi_{0,t},\psi_\C)=-i$,
which is indeed equal to $\epsilon(\la_{0,t},\psi_\R)\epsilon(\la_{1,t+1},\psi_\R)=-i$

\subsection{$\GL(1)$-twist}
Assume $F=\C$. Let $\chi_{-M,s}$ be a character on
$\C^\times$, and let $\varphi$ be an $n$-dimensional representation of
$W_\C$ as in \eqref{E:general_parameter_C}. Then the twist
$\varphi\otimes\chi_{-M,s}$ by  $\chi_{-M,s}$ is given by
\begin{equation}\label{E:twisted_general_parameter_C}
\varphi\otimes\chi_{-M,s}
=\chi_{-(N_1+M), t_1+s}\oplus\cdots\oplus\chi_{-(N_n+M), t_n+s}.
\end{equation}

Assume $F=\R$. Let $\la_{\delta, s}$ be a character on $\R^\times$. For
the 1-dimensional parameter $\la_{\ep,t}$, the twist by $\la_{\delta,s}$ is given by
\[
\la_{\ep, t}\otimes\la_{\delta,s}=
\la_{\ep+\delta\,(\operatorname{mod} 2),\, t+s-\gamma},
\;\text{ where }\;
\gamma=\begin{cases}2\quad\text{if $\ep=\delta=1$};\\
0,\quad\text{otherwise}.
\end{cases}
\]
For the 2-dimensional parameter
$\varphi_{-N,t}=\Ind_{W_\C}^{W_\R}\chi_{-N, t}$, the twisted parameter
$\varphi_{-N,t}\otimes\la_{\delta,s}$ is computed as
\begin{align*}
\varphi_{-N,t}\otimes\la_{\delta,s}
&=\Ind_{W_\C}^{W_\R}(\chi_{-N, t}\otimes(\la_{\delta, s}\circ
N_{\C/\R}))\\
&=\Ind_{W_\C}^{W_\R}(\chi_{-N, t}\otimes\chi_{0, s-\delta})\\
&=\Ind_{W_\C}^{W_\R}\chi_{-N, t+s-\delta},
\end{align*}
namely
\[
\varphi_{-N, t}\otimes\la_{\delta, s}=\varphi_{-N, t+s-\delta}.
\]
Accordingly, we have
\begin{align}
\label{E:GL(2)xGL(1)_L}L(\varphi_{-N, t}\otimes\la_{\delta, s})&=\begin{cases}
2(2\pi)^{-(t+s-\delta)}\Gamma(t+s-\delta),\quad\text{if $N\geq 0$};\\
2(2\pi)^{-(t+s-\delta-N)}\Gamma(t+s-\delta-N) \quad\text{if $N<0$};
\end{cases}\\
\label{E:GL(2)xGL(2)_ep}\epsilon(\varphi_{-N, t}\otimes\la_{\delta, s}, \psi_\R)&=-i^{|N|+1}.
\end{align}

If $\varphi:W_\R\rightarrow\GL_n(\C)$ is an $n$-dimensional
representation as
in \eqref{E:general_parameter_R}, we have
\[
\varphi\otimes\la_{\delta,  s}=
\left(\la_{\delta_1,t_1+s-\gamma_1} \oplus\cdots\oplus\la_{\delta_p,t_p+s-\gamma_p}\right)\oplus
\left(\varphi_{-N_1,u_1+s-\delta}\oplus\cdots\oplus\varphi_{-N_q, u_q+s-\delta}\right),
\]
where $\delta_i=\ep_i+\delta\pmod{2}$, and $\gamma_i=2$ if
$\ep_i=\delta=1$ and $\gamma_i=0$ otherwise.

\subsection{$\GL(2)$-twist}
Assume $F=\R$. Let $\varphi_{-N,t}$ and $\varphi_{-M,s}$ be
2-dimensional representations of $W_\R$ as above with $N, M >  0$. Then one can see
\begin{align*}
\varphi_{-N,t}\otimes\varphi_{-M,s}=&
\left(\Ind_{W_\C}^{W_\R}\chi_{-N,t}\right)\otimes
\left(\Ind_{W_\C}^{W_\R}\chi_{-M,s}\right)\\
=&\left(\Ind_{W_\C}^{W_\R}\chi_{-N,t}\cdot\chi_{-M,s}\right)\oplus
\left(\Ind_{W_\C}^{W_\R}\chi_{-N,t}\cdot\overline{\chi_{-M,s}}\right)\\
=&\left(\Ind_{W_\C}^{W_\R}\chi_{-(N+M),t+s}\right)\oplus
\left(\Ind_{W_\C}^{W_\R}\chi_{-N,t}\cdot\chi_{M,s-M}\right)\\
=&\varphi_{-(N+M), t+s}\oplus
\left(\Ind_{W_\C}^{W_\R}\chi_{-(N-M),t+s-M}\right)\\
=&\varphi_{-(N+M), t+s}\oplus\varphi_{-(N-M), t+s-M}.
\end{align*}
Of course it is possible that $N=M$, so that $\varphi_{-(N-M), t+s-M}$ is reducible.  However, let us note that for the parameter $\varphi_{0, t}$, we can check that $L(\varphi_{0,
  t})=L(\chi_{0, t})=2(2\pi)^{-t}\Gamma(t)$, since it is indeed equal to
$L(\la_{0,t})L(\la_{1,t+1})=\pi^{-t/2}\Gamma(\frac{t}{2})\cdot\pi^{-(t+1)/2}\Gamma(\frac{t+1}{2})$
by the duplication formula. Also
one can check that
$\epsilon(\varphi_{0,t})=-i\cdot\epsilon(\chi_{0,t},\psi_\C)=-i$,
which is indeed equal to $\epsilon(\la_{0,t},\psi_\R)\epsilon(\la_{1,t+1},\psi_\R)=-i$.  Then, coupled together with the identity 
\[
\varphi_{-(N-M), t+s-M}=\varphi_{-(M-N), t+s-N}
\]
(by \eqref{E:N>0}), we obtain
\begin{align}
\label{E:GL(2)xGL(2)_L}
  L(\varphi_{-N,t}\otimes\varphi_{-M,s})&=4(2\pi)^{-2(t+s)+\min\{N, M\}}\cdot\Gamma(t+s)\cdot
\Gamma(t+s-\min\{N, M\})\\
\label{E:GL(2)xGL(2)_ep}\epsilon(\varphi_{-N,t}\otimes\varphi_{-M,s})&=(-1)^{\max\{N, M\}}.
\end{align}

\subsection{Local Langlands correspondence for $\GL_n(F)$}
By the archimedean local Langlands correspondence, originally
established by Langlands (\cite{Langlands}), there is a one-to-one
correspondence between the set $\Irr_n$ and the set $\Phi_n$ of all continuous
$n$-dimensional representations of $W_F$. If $\pi_1\in\Irr_n$ and
$\pi_2\in\Irr_t$ correspond to $\varphi_1\in\Phi_n$ and $\varphi_2\in\Phi_t$,
respectively, then the local factors are defined to be
\begin{align*}
L(s,\pi_1\times\pi_2)&=L(\varphi_1\otimes\varphi_2|\cdot|_F^s)\\
\epsilon(s,\pi_1\times\pi_2,\psi_F)&=
\epsilon(\varphi_1\otimes\varphi_2|\cdot|_F^s,\psi_F),
\end{align*}
where $|\cdot|_{\C}=\|\cdot\|$ and $|\cdot|_{\R}=|\cdot|$. Further if
$\varphi\in\Irr_n$ corresponds to $\varphi_\pi\in\Phi_n$ and if $\omega_\pi$
is the central character of $\pi$, we have
\[
\omega_\pi=\det(\varphi_\pi).
\]
In particular, if $\pi$ is a representation of $\GL_2(\R)$ such that
$\varphi_\pi=\varphi_{-N,t}$, then $\omega_\pi$ is as in \eqref{E:central_character}.

%%%%%%%%%%%%%%%%%%%%%%%%%%%%%%%%%%%%%%%%%%%%%%%%%%%%%%%%%%%%%%%%%%%

\section{A partial order}

%%%%%%%%%%%%%%%%%%%%%%%%%%%%%%%%%%%%%%%%%%%%%%%%%%%%%%%%%%%%%%%%%%%

In this section we will define a certain partial order on complex numbers,
which will play a crucial role in our proof. 

\subsection{Definition}
We define a partial order $\preceq$ on the set $\C$ of
complex numbers as follows: For $t_1,t_2\in\C$, 
\[
t_1\preceq t_2\quad\text{if}\quad t_2-t_1\in\Z^{\geq 0},
\]
and otherwise we say they are \emph{incomparable}. We use the symbol $\prec$ for
strict inequality, namely
\[
t_1\prec t_2\quad\text{if}\quad t_1\preceq t_2\text{ but } t_1\neq t_2.
\]
Then one can check that $t_1\prec t_2$ if
and only if
\begin{equation*}
\{\text{poles of $\Gamma(s+t_1)$}\}\supsetneq \{\text{poles of $\Gamma(s+t_2)$}\},
\end{equation*}
and $t_1$ and $t_2$ are incomparable if and only if
\begin{equation*}
\{\text{poles of $\Gamma(s+t_1)$}\}\cap \{\text{poles of $\Gamma(s+t_2)$}\}=\emptyset.
\end{equation*}

If $A$ is a finite set of complex numbers and $t\in A$
is a minimal element in $A$ (\ie whenever $t'\in A$ is such that
$t'\preceq t$, we have $t'=t$), we say $t$ is {\it minimal in}
$A$. Of course, a minimal element might not
be unique.

Let $F(s)$ be a function on $\C$.  If $F(s)$ has a pole at $s=t$ and
if it is maximal among the poles of $F(s)$ with respect to $\preceq$
(\ie if $F(s)$ has a pole at $s=t'$ and
$t\preceq t'$, then $t=t'$), then we call the pole at $s=t$ a {\it
  maximal pole} of $F(s)$. In particular, for a fixed $t\in\C$, the gamma function
$\Gamma(s+t)$ has a unique maximal pole at $s=-t$. More generally, a
product $\prod_{i=1}^n\Gamma(s+t_i)$ of gamma functions has a maximal
pole precisely at $s=-t_i$ where $t_i$ is a minimal
element in the set $\{t_1,\dots,t_n\}$. Of course, again, a maximal
pole is not necessarily unique in general.

\subsection{A lemma}
The following lemma will be repeatedly used throughout.
\begin{Lem}\label{L:Lemma}
Let $t_1,\dots, t_m, t'_1,\dots, t'_{m'}\in\C$. Suppose
\begin{equation}\label{E:Lemma}
F(s)\cdot\prod_{i=1}^m\Gamma(s+t_i)=\prod_{j=1}^{m'}\Gamma(s+t'_j)
\end{equation}
as functions in $s$, where $F(s)$ is a function on $\C$ whose zeros and poles
do not interfere with the poles of the $\Gamma(s+t_i)$'s and
$\Gamma(s+t'_i)$'s. Then $m=m'$, and further for each $i\in\{1,\dots,m\}$,
there is a corresponding $j\in\{1,\dots,m'\}$ such that
$t_i=t'_j$, namely
\[
\{t_1,\dots,t_m\}=\{t'_1,\dots,t'_{m'}\}
\]
 as multisets. Accordingly, $F\equiv 1$.
\end{Lem}
\begin{proof}
This is almost immediate if one looks at the maximal poles of both
sides. Namely, if $t_k$ is minimal in $\{t_1,\cdots,t_m\}$, the left
hand side has a maximal pole at $s=-t_k$, and hence the right hand
side has a maximal pole at $s=-t_k$. But
each maximal pole of the right hand side occurs at $s=-t'_l$ for some
$t'_l$ which is minimal in $\{t'_1,\dots,t'_{m'}\}$. Hence $t_k=t'_l$ for some $l$. After reordering
the indices, if necessary, we may assume $t_1=t'_1$. Hence
\eqref{E:Lemma} is written as
\[
F(s)\cdot\prod_{i=2}^m\Gamma(s+t_i)=\prod_{j=2}^{m'}\Gamma(s+t'_j).
\]
By arguing inductively, one proves the lemma.
\end{proof}

%%%%%%%%%%%%%%%%%%%%%%%%%%%%%%%%%%%%%%%%%%%%%%%%%%%%%%%%%%%%%%%%%%%

\section{Proof for complex case}

%%%%%%%%%%%%%%%%%%%%%%%%%%%%%%%%%%%%%%%%%%%%%%%%%%%%%%%%%%%%%%%%%%%

We will prove our main theorem for $F=\C$. Let $\pi$ and $\pi'$ be
infinitesimal equivalence classes of irreducible admissible
representations of $\GL_n(\C)$, and assume that their corresponding
Langlands parameters $\varphi$ and $\varphi'$ are given by
\begin{align*}
\varphi&=\chi_{-N_1,t_1}\oplus\cdots\oplus\chi_{-N_n, t_n}\\
\varphi'&=\chi_{-N'_1,t'_1}\oplus\cdots\oplus\chi_{-N'_n, t'_n}.
\end{align*}
The assertion $L(s, \pi\times\tau)=L(s, \pi'\times\tau)$ for all
$\tau\in\Irr_1$ is the same as
\begin{equation}\label{E:L-equality-C0}
L(\varphi\otimes\chi_{-M, s})=L(\varphi'\otimes\chi_{-M, s})
\end{equation}
for all $M\in\Z$ and $s\in\C$. As in
\eqref{E:twisted_general_parameter_C}, we have
\begin{align*}
\varphi\otimes\chi_{-M,s}&=\chi_{-(N_1+M),t_1+s}\oplus\cdots\oplus\chi_{-(N_n+M), t_n+s}\\
\varphi'\otimes\chi_{-M,s}&
=\chi_{-(N'_1+M),t'_1+s}\oplus\cdots\oplus\chi_{-(N'_n+M),t'_n+s},
\end{align*}
and hence the equality \eqref{E:L-equality-C0} implies
\begin{equation}\label{E:L-equality-C}
\prod_{i=1}^nL(\chi_{-(N_i+M),t_i+s})
=\prod_{j=1}^nL(\chi_{-(N'_j+M),t'_j+s})
\end{equation}
for all $M\in\Z$ and $s\in\C$.

We will show \eqref{E:L-equality-C} implies $\varphi=\varphi'$. But
once we show that $\varphi$ and $\varphi'$ have
a common constituent, the theorem follows by induction. Namely, it
suffices to prove
\begin{equation}\label{E:goal}
\chi_{-N_i, t_i}=\chi_{-N'_j, t'_j}
\end{equation}
for some $i, j\in\{1,\dots,n\}$.

The basic idea of our proof is that we will examine ``maximal poles'' of
both sides of \eqref{E:L-equality-C} for
different choices of $M$. Our proof has three major steps, which is
outlined as follows. In Step 1, we will show that $\{t_1,\dots, t_n\}=\{t'_1,\dots,
t'_n\}$ as multisets. In Step 2, we will show that
$\min\{N_1,\dots,N_n\}=\min\{N'_1,\dots,N'_n\}$. Finally in Step 3, we
will finish up the proof of the theorem. In each of the steps we
will choose an appropriate $M$.

\subsection{Step 1}
Let us choose $M$ large enough so that $N_i+M> 0$ and
$N'_j+M>0$
for all $i, j\in\{1,\cdots,n\}$. By \eqref{E:L-factor-C}, we can write
\eqref{E:L-equality-C} as
\[
F(s)\prod_{i=1}^n\Gamma(s+t_i)=\prod_{j=1}^n\Gamma(s+t'_j)
\]
for some function $F(s)$ without a pole or zero. Hence by Lemma
\ref{L:Lemma}, we have
\[
\{t_1,\dots, t_n\}=\{t'_1,\dots, t'_n\}
\]
as multisets. 

\subsection{Step 2}
Let us first note that we can now reorder the indices $i$
and $j$ in such a way that
\begin{equation}\label{E:relation2}
t_1=t'_1, t_2=t'_2,\dots, t_n=t'_n.
\end{equation}
Let
\begin{align*}
N_{\min}:=&\min\{N_1,\dots,N_n\}\\
N'_{\min}:=&\min\{N'_1,\dots,N'_n\}.
\end{align*}
As we mentioned above, we will show $N_{\min}=N'_{\min}$ in this step. 
First assume $N'_{\min}<N_{\min}$. Choose
$M=-N'_{\min}-1$. Then \eqref{E:L-equality-C} is written as
\begin{equation}\label{E:L-equality-C1}
F(s)\prod_{i=1}^n\Gamma(s+t_i)=\prod_{j=1}^n\Gamma(s+t'_j+\ep_j),
\end{equation}
where 
\[
\ep_j=\begin{cases}1,\quad\text{if $N'_j=N'_{\min}$};\\
0,\quad\text{if $N'_j>N'_{\min}$},
\end{cases}
\]
and $\ep_j\neq 0$ for at least one $j$.
Now the left hand side of \eqref{E:L-equality-C1} has a maximal pole
at $s=-t_k$ for some $t_k$ which is minimal in
$\{t_1,\dots,t_n\}$. Then $t_k=t'_l+\ep_l$ for some $l$ and
$t'_l+\ep_l$ is minimal in
$\{t'_1+\ep_1,\dots,t'_n+\ep_n\}$. Assume $\ep_l=1$. Then we have
$t'_l+\ep_l=t_l+1$ by \eqref{E:relation2}, which implies
$t_l+1=t_k$. But then $t_l\prec t_k$, which
contradicts the minimality of $t_k$. Hence
$\ep_l=0$, namely $t_k=t'_l$. Then
\eqref{E:L-equality-C1} is now written as
\begin{equation}\label{E:L-equality-C2}
F(s)\prod_{i\neq k}\Gamma(s+t_i)=\prod_{j\neq l}\Gamma(s+t'_j+\ep_j).
\end{equation}
by cancelling $\Gamma (s+t_k)$ from the left and $\Gamma (s+t'_l)$ from
the right. By arguing inductively, one can
show that $\ep_j=0$ for all $j\in\{1,\dots,n\}$ which is a
contradiction. Hence we must have $N'_{\min}\geq N_{\min}$. By
symmetry, we can show that $N'_{\min}\leq N_{\min}$. Hence
$N_{\min}=N'_{\min}$.

\subsection{Step 3}
Finally, we will show \eqref{E:goal} for some $i$ and $j$, which will
complete the proof of the theorem. For this
purpose, let us reorder the indices so that 
\begin{equation}\label{E:relation1}
N_1=N_2=\cdots=N_d<N_{d+1}\leq\cdots\leq N_n,
\end{equation}
while still keeping the relation \eqref{E:relation2}. Such indexing is
certainly possible. Now let us twist $\varphi$ and $\varphi'$ by $\chi_{-M,s}$ with
$M=-N_1-1=-N'_{\min}-1$. Then \eqref{E:L-equality-C} is now written as
\begin{equation}\label{E:L-equality-C3}
F(s)\prod_{i=1}^d\Gamma(s+t_i+1)\cdot\prod_{i=d+1}^n\Gamma(s+t_i)
=\prod_{j=1}^n\Gamma(s+t'_j+\ep_j)
\end{equation}
where $\ep_j$ is as before and $\ep_j=1$ for at least one $j$. Let
\begin{align*}
A&:=A_1\cup A_2, \quad A_1:=\{t_i+1:i\in\{1,\dots,d\}\},\quad A_2:=\{t_i:i\in\{d+1,\dots,n\}\}\\
B&:=\{t'_j+\ep_j: j\in\{1,\dots,n\}\}.
\end{align*}
Of course by Lemma \ref{L:Lemma}, we know $A=B$ as multisets. 
We will again look at ``maximal poles'' of
\eqref{E:L-equality-C3}, which correspond to minimal elements of $A=B$.
We consider
the following two cases, depending on whether a minimal element is in
$A_1$ or $A_2$.

First assume $t_k+1\in A_1$ with $k\leq d$ is minmal in
$A$. By $A=B$, we have $t_k+1=t'_l+\ep_l$ for some $l$, which is
minimal in $B$. Assume $\ep_l=1$. Then this
means $N'_l=N_1=N_k$ and $t_k=t'_l$. Hence we have
\eqref{E:goal} with $i=k$ and $j=l$. Assume $\ep_l=0$. First assume $\ep_k=0$. Then
$t'_k+\ep_k=t_k=t'_l+\ep_l-1$, namely $t'_k+\ep_k\prec
t'_l+\ep_l$, which contradicts the minimality of
$t'_l+\ep_l$. Hence $\ep_k=1$, which implies $N'_k=N_1=N_k$. Since
$t_k=t'_k$ by \eqref{E:relation2}, the equality \eqref{E:goal} is satisfied with $i=j=k$.

Next assume that $t_k\in A_2$ with $k>d$ is minmal in
$A$. By $A=B$,
we know that $t_k=t'_l+\ep_l$ for some $l$, which is minimal in
$B$. Assume $\ep_l=1$, so $t_k=t'_l+1=t_l+1$. Then
$t_l\prec t_k$. Since $t_k$ is minimal in
$A$, we must have $l\leq d$. Hence
$N_l=N_1$. Also since $\ep_l=1$ implies $N'_l=N_1$, we have
$N'_l=N_l$. Hence, we have
\eqref{E:goal} with $i=j=l$. Finally assume 
$\ep_l=0$, \ie $t_k=t'_l$. Then we
can reorder the indices $j$ for the multiset $B$ by swapping $k$ and
$l$ without changing the relations \eqref{E:relation2} and
\eqref{E:relation1}. So we can cancel $\Gamma(s+t_k)$ from both
sides of \eqref{E:L-equality-C3}. We can now argue
inductively by repeating the above arguments until we obtain
\eqref{E:goal}.

The proof is complete.

%%%%%%%%%%%%%%%%%%%%%%%%%%%%%%%%%%%%%%%%%%%%%%%%%%%%%%%%%%%%%%%%%%%

\section{Proof for real case}

%%%%%%%%%%%%%%%%%%%%%%%%%%%%%%%%%%%%%%%%%%%%%%%%%%%%%%%%%%%%%%%%%%%

We will prove our main theorem for $F=\R$. Assume $\pi, \pi'\in\Irr_n$ are
such that $L(s, \pi\times\tau)=L(s,\pi'\times\tau)$ for all
$\tau\in\Irr_t$ with $t=1,2$. Namely if
$\varphi$ and $\varphi'$ are the corresponding Langlands parameters,
we have
\begin{equation}\label{E:character-twist}
L(\varphi\otimes\la_{\ep,s})=L(\varphi'\otimes\la_{\ep,s})
\end{equation}
for all $\ep\in\{0,1\}$ and $s\in\C$, and 
\begin{equation}\label{E:GL(2)-twist}
L(\varphi\otimes\varphi_{-M,s})=L(\varphi'\otimes\varphi_{-M,s})
\end{equation}
for all $M\in\Z$ and $s\in\C$. 

Let us write 
\begin{align*}
\varphi&=\left(\la_{\ep_1,t_1}\oplus\cdots\oplus\la_{\ep_p,t_p}\right)\oplus
\left(\varphi_{-N_1,u_1}\oplus\cdots\oplus\varphi_{-N_q, u_q}\right)\\
\varphi'&=\left(\la_{\ep'_1,t'_1}\oplus\cdots\oplus\la_{\ep'_{p'},t'_{p'}}\right)\oplus
\left(\varphi_{-N'_1,u'_1}\oplus\cdots\oplus\varphi_{-N'_{q'} u'_{q'}}\right)
\end{align*}
where we may assume $N_i>0$ and $N'_j>0$ for all $i, j$. Note that
the numbers of 1-dimensional constituents for $\varphi$ and $\varphi'$
are, respectively, $p$ and $p'$, and those for 2-dimensional ones are,
respectively, $q$ and $q'$, so $n=p+2q=p'+2q'$.\\

The proof has three steps, and the basic philosophy is the same as the
complex case in that we choose appropriate twists and examine
``maximal poles''. In Step 1, we will show $\{u_1, ..., u_q\}=\{u'_1, ..., u'_{q'} \}$ and 
$\{t_1-\ep_1, ..., t_p-\ep_p\}=\{t'_1-\ep'_1, ..., t'_{p'}-\ep'_{p'}\}$ as multisets, which also implies $p=p'$ and
$q=q'$. In Step 2, we will show
the 1-dimensional constituents are equal. Finally in Step 3, we will show
the 2-dimensional constituents are equal.

\subsection{Step 1}
Consider the twist by $\varphi_{-M, s}$ with $M\geq 0$. Since
all the $N_i$'s and $N'_j$'s are positive, \eqref{E:GL(2)xGL(1)_L} and
\eqref{E:GL(2)xGL(2)_L} give
\begin{align*}
L(\varphi\otimes\varphi_{-M,  s})&=F(s)
\prod_{i=1}^{p}\Gamma(s+t_i-\ep_i)\cdot
\prod_{i=1}^{q}\Gamma(s+u_i)\cdot\Gamma(s+u_i-\min\{M, N_i\})\\
L(\varphi'\otimes\varphi_{-M,  s})&=F'(s)
\prod_{j=1}^{p'}\Gamma(s+t'_j-\ep'_j)\cdot
\prod_{j=1}^{q'}\Gamma(s+u'_j)\cdot\Gamma(s+u'_j-\min\{M, N'_j\}),
\end{align*}
where $F(s)$ and $F'(s)$ are functions without a zero or pole. And the
equality \eqref{E:GL(2)-twist} implies those two are equal. For the
choices $M=0$ (so $\min\{M, N_i\} =\min\{M, N'_j\}=0$) and $M=1$
(so $\min\{M, N_i\} =\min\{M, N'_j\}=1$), Lemma \ref{L:Lemma} implies
\begin{align}\label{E:multisets_t_u-0}
&\{t_1-\ep_1,\dots,t_p-\ep_p\}\cup\{u_1,\dots,u_q\}\cup\{u_1,\dots,u_q\}\\
\notag=&\{t'_1-\ep'_1,\dots,t'_{p'}-\ep'_{p'}\}\cup\{u'_1,\dots,u'_{q'}\}\cup\{u'_1,\dots,u'_{q'}\},
\end{align}
and
\begin{align}\label{E:multisets_t_u-1}
&\{t_1-\ep_1,\dots,t_p-\ep_p\}\cup\{u_1,\dots,u_q\}\cup\{u_1-1,\dots,u_q-1\}\\
\notag=&\{t'_1-\ep'_1,\dots,t'_{p'}-\ep'_{p'}\}\cup\{u'_1,\dots,u'_{q'}\}
\cup\{u'_1-1,\dots,u'_{q'}-1\}
\end{align}
as multisets. In what follows, by using \eqref{E:multisets_t_u-0} and \eqref{E:multisets_t_u-1}, we will show
\begin{align}\label{E:multisets_t-ep}
\{t_1-\ep_1,\dots,t_p-\ep_p\}&=\{t'_1-\ep'_1,\dots,t'_{p'}-\ep'_{p'}\}\\
\label{E:multisets_u}
\{u_1,\dots,u_q\}&=\{u'_1,\dots,u'_{q'}\}
\end{align}
as multisets, which would also imply $p=p'$ and $q=q'$.

The basic idea is to argue inductively by comparing a minimal element in both sides of
\eqref{E:multisets_t_u-0}, and after each step we will shrink the
multisets in both \eqref{E:multisets_t_u-0} and
\eqref{E:multisets_t_u-1} by eliminating the minimal element.

First assume $t_k-\ep_k$ is minimal in the left hand side of
\eqref{E:multisets_t_u-0}. Then either
$t_k-\ep_k=t'_l-\ep'_{l'}$ for some $l$, or $t_k-\ep_k=u'_l$. If the
former is the case, \ie
\[
t_k-\ep_k=t'_l-\ep'_l,
\]
we can eliminate $t_k-\ep_k$ from the left hand
sides of both \eqref{E:multisets_t_u-0} and \eqref{E:multisets_t_u-1}
and eliminate $t'_l-\ep'_l$ from the right hand
sides of both \eqref{E:multisets_t_u-0} and \eqref{E:multisets_t_u-1},
which allows us to shrink the multisets in those two equalities and obtain
\begin{align*}
&\{t_i-\ep_i: i\neq k\}\cup\{u_1,\dots,u_q\}\cup\{u_1,\dots,u_q\}\\
\notag=&\{t'_j-\ep'_j: j\neq l\}\cup\{u'_1,\dots,u'_{q'}\}\cup\{u'_1,\dots,u'_{q'}\},
\end{align*}
and
\begin{align*}
&\{t_i-\ep_i:i\neq k\}\cup\{u_1,\dots,u_q\}\cup\{u_1-1,\dots,u_q-1\}\\
\notag=&\{t'_j-\ep'_j: j\neq l\}\cup\{u'_1,\dots,u'_{q'}\}
\cup\{u'_1-1,\dots,u'_{q'}-1\}.
\end{align*}
If the latter is the case, namely if $t_k-\ep_k=u'_l$, then 
by \eqref{E:multisets_t_u-1}, 
$u'_l-1=t_k-\ep_k-1$ is a minimal element in both sides of
\eqref{E:multisets_t_u-1}. But since $t_k-\ep_k$ is minimal in the
left hand side of \eqref{E:multisets_t_u-0}, $t_k-\ep_k-1$, which is
strictly smaller than $t_k-\ep_k$, has to be among
$\{u_1-1,\dots,u_q-1\}$. Namely,
$t_k-\ep_k-1=u_m-1$ for some $m$, which implies 
\[
u_m=u'_l.
\] 
Hence by
eliminating $u_m$ and $u'_l$, we can shrink both \eqref{E:multisets_t_u-0} and
\eqref{E:multisets_t_u-1} and obtain
\begin{align*}
&\{t_1-\ep_1,\dots,t_p-\ep_p\}\cup\{u_i:i\neq m\}\cup\{u_i:i\neq m\}\\
\notag=&\{t'_1-\ep'_1,\dots,t'_{p'}-\ep'_{p'}\}\cup\{u'_j:j\neq
l\}\cup\{u'_j:j\neq l\},
\end{align*}
and
\begin{align*}
&\{t_1-\ep_1,\dots,t_p-\ep_p\}\cup\{u_i:i\neq m\}\cup\{u_i-1:i\neq m\}\\
\notag=&\{t'_1-\ep'_1,\dots,t'_{p'}-\ep'_{p'}\}\cup\{u'_j:j\neq
l\}\cup\{u'_j-1:j\neq l\}.
\end{align*}

Next assume $u_k$ is a minimal element in the left hand side of
\eqref{E:multisets_t_u-0}. Then either $u_k=u'_l$ or
$u_k=t'_l-\ep_l'$ for some $l$. If the former is the case, one can shrink both sides
of \eqref{E:multisets_t_u-0} and \eqref{E:multisets_t_u-1} as
above. If the latter is the case, by arguing in the same way one can
obtain $u_k-1=u'_m-1$ for some $m$, which gives $u_k=u'_m$ and allows
us to shrink the multisets in \eqref{E:multisets_t_u-0} and
\eqref{E:multisets_t_u-1} as before.

To summarize, at each step one can obtain either
$t_i-\ep_i=t'_j-\ep'_j$ or $u_i=u'_j$ for some $i, j$, and by
eliminating this element from \eqref{E:multisets_t_u-0} and
\eqref{E:multisets_t_u-1}, one can shrink the multisets from those two
equalities. But the resulting equalities are of the same form, and
hence we can apply the same argument to those two shrunken
equalities. By repeating the process, one can obtain
\eqref{E:multisets_t-ep} and \eqref{E:multisets_u}, which also imply $p=p'$ and $q=q'$.

\subsection{Step 2}
We will show the 1-dimensional constituents of $\varphi$ and
$\varphi'$ are equal. By twisting by $\la_{0,s}$, we have
\begin{align*}
\varphi\otimes\lambda_{0, s}
&=\left(\lambda_{\ep_1, t_1+s}\oplus\cdots\oplus\lambda_{\ep_p,
  t_p+s}\right)
\oplus\left(\varphi_{-N_1, u_1+s}\oplus\cdots\oplus \varphi_{-N_q,
  u_q+s}\right)\\
\varphi'\otimes\lambda_{0, s}
&=\left(\lambda_{\ep'_1, t'_1+s}\oplus\cdots\oplus\lambda_{\ep'_p,
  t'_p+s}\right)
\oplus\left(\varphi_{-N'_1, u'_1+s}\oplus\cdots\oplus \varphi_{-N'_q,
  u'_q+s}\right),
\end{align*}
so the equality
$L(\varphi\otimes\la_{0,s})=L(\varphi'\otimes\la_{0,s})$
implies
\[
F(s)\prod_{i=1}^{p}\Gamma\left(\frac{s+t_i}{2}\right)\cdot\prod_{i=1}^q\Gamma(s+u_i)
=\prod_{j=1}^{p}\Gamma\left(\frac{s+t'_j}{2}\right)\cdot\prod_{j=1}^q\Gamma(s+u'_j),
\]
where $F(s)$ is a function without a zero or pole. Since we know
$\{u_1, ..., u_q\}=\{u'_1, ..., u'_q\}$ as multisets, we have
\[
F(s)\prod_{i=1}^{p}\Gamma\left(\frac{s+t_i}{2}\right)
=\prod_{j=1}^{p}\Gamma\left(\frac{s+t'_j}{2}\right),
\]
from which we have
\begin{equation}\label{E:multisets_t}
\{t_1,\dots,t_p\}=\{t'_1,\dots,t'_p\}
\end{equation}
as multisets, by Lemma \ref{L:Lemma}. 

By using \eqref{E:multisets_t-ep} and \eqref{E:multisets_t},
we will show that the 1-dimensional
constituents of $\varphi$ and $\varphi'$ are equal. But of course, by induction, it
suffices to show $\la_{\ep_i,t_i}=\la_{\ep'_j,t'_j}$ for some $i$ and
$j$. We will show this by examining minimal elements in those
multisets.

First, let us note that since we have \eqref{E:multisets_t}, we may
reorder the indices so that $t_i=t'_i$ for all
$i\in\{1,\dots,n\}$. Let $t_k=t'_k$ be minimal in
$\{t_1,\dots,t_p\}$. Clearly if $\ep_k=\ep'_k$, then $\la_{\ep_k,
  t_k}=\la_{\ep'_k, t'_k}$, so we are done. So let us assume
$\ep_k\neq\ep'_k$. By symmetry, we may assume $\ep_k=1$ and
$\ep'_k=0$. Note, then, that $t_k-\ep_k$ is
minimal in $\{t_1-\ep_1,\dots,t_p-\ep_p\}$. By
\eqref{E:multisets_t-ep}, we must have $t_k-\ep_k=t'_l-\ep'_l$ for some
$l$. Assume $\ep'_l= 0$. Then $t_k-1=t'_l$, which implies
$t'_l=t_l\prec t_k$. This contradicts the minimality of
$t_k$. Hence $\ep'_l=1=\ep_k$. Then $t_k-\ep_k=t'_l-\ep'_l$ implies
$t_l=t'_k$. Namely we have $\la_{\ep_k,t_k}=\la_{\ep'_l, t'_l}$.

 This completes the proof that $\varphi$ and $\varphi'$
have the same 1-dimensional constituents.

\subsection{Step 3}
Finally, we will show that the 2-dimensional constituents are
equal. Since we have already shown the 1-dimensional constituents are
equal, we may assume
\begin{align*}
\varphi&=\varphi_{-N_1,u_1}\oplus\cdots\oplus\varphi_{-N_q, u_q}\\
\varphi'&=\varphi_{-N'_1,u'_1}\oplus\cdots\oplus\varphi_{-N'_{q'} u'_{q}}.
\end{align*}
Then the proof is very similar to the complex case. First note that
since we already have \eqref{E:multisets_u}, we can
reorder the indices in such a way that
\begin{equation}\label{E:relation2-R}
u_1=u'_1, u_2=u'_2,\dots, u_n=u'_n.
\end{equation}
Let us write
\begin{align*}
N_{\min}&=\min\{N_1,\dots, N'_q\}\\
N'_{\min}&=\min\{N'_1,\dots, N'_q\},
\end{align*}
and we will show $N'_{\min}=N_{\min}$. Assume
$N'_{\min}<N_{\min}$, and consider the twist by $\varphi_{-M, s}$
with $M=N_{\min}$. Then the equality
$L(\varphi\otimes\varphi_{-N_{\min},s})=L(\varphi'\otimes\varphi_{-N_{\min},s})$ is
written as
\begin{equation}\label{E:M=M'_j}
F(s)\prod_{i=1}^q\Gamma(s+u_i-N_{\min})=\prod_{j=1}^q\Gamma(s+u'_j-M'_j),
\quad M'_j=\min\{N'_j, N_{\min}\},
\end{equation}
for some $F(s)$ without a zero or a pole, where we used \eqref{E:GL(2)xGL(2)_L}.
Let $J=\{j\in\{1,\dots,n\}: N'_j<N_{\min}\}$. By our assumption, $J$ is not
empty. Since we have \eqref{E:relation2-R}, the equality
\eqref{E:M=M'_j} simplifies to
\begin{equation}\label{E:M=M'_j2}
F(s)\prod_{i\in J}\Gamma(s+u_i-N_{\min})=\prod_{j\in J}\Gamma(s+u'_j-N'_j).
\end{equation}
By Lemma \ref{L:Lemma}, we have
\[
\{u_i-N_{\min}: i\in J\}=\{u'_j-N'_j: j\in J\}
\]
as multisets. Now suppose $u_k-N_{\min}$ is minimal in the left hand
side, which is the same as supposing 
$u_k$ is minimal in $\{u_i: i\in J\}$. There exists $l\in J$
such that $u_k-N_{\min}=u'_l-N'_l$, which gives
$u_k=u'_l+(N_{\min}-N'_l)$.  Since $N_{\min}-N'_l>0$, this implies
$u_l=u'_l\prec u_k$, which contradicts the
minimality of $u_k$. Hence $N'_{\min}\geq N_{\min}$. By symmetry, we
have $N'_{\min}\leq N_{\min}$, namely $N'_{\min}=N_{\min}$.

Then we can proceed analogously to Step 3 of the complex case. Namely,
first let us reorder the indices so that in addition to
\eqref{E:relation2-R} we have
\begin{equation}\label{E:relation1-R}
0<N_1=\cdots=N_{d}<N_{d+1}\leq\cdots\leq N_q.
\end{equation}
By twisting by $\varphi_{-M,s}$ with $M=N_1+1$, we have
\begin{equation}\label{E:L-equality-R3}
F(s)\prod_{i=1}^d\Gamma(s+u_i-N_1)\prod_{i=d+1}^q\Gamma(s+u_i-N_1-1)
=\prod_{j=1}^q\Gamma(s+u_i-N_1-1+\ep_j),
\end{equation}
where
\[
\ep_j=\begin{cases}1,\quad\text{if $N'_j=N_1=N'_{\min}$}\\
0,\quad\text{if $N'_j>N_1$}.
\end{cases}
\]
By shifting $s\mapsto s+N_1+1$, the equality \eqref{E:L-equality-R3}
becomes
\begin{equation}\label{E:L-equality-R3-1}
F(s+N_1+1)\prod_{i=1}^d\Gamma(s+u_i+1)\cdot\prod_{i=d+1}^q\Gamma(s+u_i)
=\prod_{j=1}^q\Gamma(s+u_i+\ep_j).
\end{equation}
This is exactly the same as \eqref{E:L-equality-C3} of the complex
case. Hence by arguing in the same way, we can show that there exists
a pair of indices $i$ and $j$ such that $u_i=u'_j$ and
$N_i=N'_j$. Hence for such indices, we have $\varphi_{-N_i,
  u_i}=\varphi_{-N'_j, u'_j}$, namely $\varphi$ and $\varphi'$ have a
common constituent. By arguing inductively, we have
$\varphi=\varphi'$. 

The proof is complete.

%%%%%%%%%%%%%%%%%%%%%%%%%%%%%%%%%%%%%%%%%%%%%%%%%%%%%%%%%%%%%%%%%%%

\section{Low rank cases}

%%%%%%%%%%%%%%%%%%%%%%%%%%%%%%%%%%%%%%%%%%%%%%%%%%%%%%%%%%%%%%%%%%%

In this section, we will consider low rank cases for $F=\R$. In
particular, we will show that twisting by $\GL(2)$ is sharp in
that one cannot characterize the set $\Irr_n$ only by $\GL(1)$ twists
even for $n=2$. Yet, we will show that for $n=2, 3$, we can
characterize it if we  assume the central characters are equal.
But for $n\geq 4$, even with the central character
assumption, the $\GL(1)$-twist is not enough.

\subsection{Case for $\GL_2(\R)$ and $\GL_3(\R)$}
We will prove the following which characterizes the sets $\Irr_2$ and
$\Irr_3$ only
by $\GL(1)$-twists and the central character assumption.
\begin{Prop}
Let $n=2$ or $3$.
Assume $\pi, \pi'\in\Irr_n$ are such that
\[
L(s, \pi\times\la)=L(s, \pi'\times\la)
\]
for all characters $\la$ on $\R^\times$. Further assume that the
central characters of $\pi$ and $\pi'$ are equal. Then $\pi=\pi'$.
\end{Prop}
\begin{proof}
We will give a proof only for
$n=3$. (The cae for $n=2$ is simpler and left to the reader.) 
Assume $\varphi$ and $\varphi'$ are the local Langlands
parameters for $\pi$ and $\pi'$.

First assume
\begin{align}
\label{E:1+1+1}\varphi&=\la_{\ep_1, t_1}\oplus\la_{\ep_2, t_2}\oplus\la_{\ep_3, t_3}\\
\label{E:2+1}\varphi'&=\varphi_{-N', u'}\oplus\la_{\ep', t'}
\end{align}
with $N'>0$, and we will derive a contradiction. First note that the
assumption on the central characters implies
\begin{equation}\label{E:central}
-(\ep_1+\ep_2+\ep_3)+t_1+t_2+t_3=2u'-N'-\ep'+t',
\end{equation}
where we used \eqref{E:central_character} to obtain the right hand side.
Next note that for each $\la_{\delta, s}$ we have 
\begin{align*}
\varphi\otimes\la_{\delta, s}
&=\la_{\delta_1, t_1+s-\gamma_1}\oplus\la_{\delta_2,
  t_2+s-\gamma_2}\oplus \la_{\delta_3, t_3+s-\gamma_3}\\
\varphi'\otimes\la_{\delta,s}
&=\varphi_{-N',u'+s-\delta}\oplus\la_{\delta', t'+s-\gamma'}
\end{align*}
where $\delta_i=\ep_i+\delta\pmod{2}$ and $\gamma_i=2$ (resp.\ 0) if
$\ep_i=\delta=1$ (resp.\ otherwise), and similarly for $\delta'$ and
$\gamma'$.
Now the equality of the $L$-functions implies
\[
F(s)\Gamma\left(\frac{s+t_1-\gamma_1}{2}\right)\cdot
  \Gamma\left(\frac{s+t_2-\gamma_2}{2}\right)\cdot
\Gamma\left(\frac{s+t_3-\gamma_3}{2}\right)
=\Gamma(s+u'-\delta)\cdot \Gamma\left(\frac{s+t'-\gamma'}{2}\right)
\]
for some $F(s)$ without a zero or pole, which is, by the duplication formula, written as
\begin{align*}
F(s)&\Gamma\left(\frac{s+t_1-\gamma_1}{2}\right)\cdot
  \Gamma\left(\frac{s+t_2-\gamma_2}{2}\right)\cdot
\Gamma\left(\frac{s+t_3-\gamma_3}{2}\right)\\
=&\Gamma\left(\frac{s+u'-\delta}{2}\right)\cdot
\Gamma\left(\frac{s+u'-\delta+1}{2}\right)\cdot
\Gamma\left(\frac{s+t'-\gamma'}{2}\right)
\end{align*}
for some other $F(s)$. By Lemma \ref{L:Lemma}, we have
\begin{equation}\label{E:multisets_GL(3)}
\{t_1-\gamma_1, t_2-\gamma_2, t_3-\gamma_3\}
=\{u'-\delta, u'-\delta+1, t'-\gamma'\}
\end{equation}
as multisets. First choose $\delta=0$, which forces
$\gamma_1=\gamma_2=\gamma_3=\gamma'=0$, and hence we have
\[
\{t_1, t_2, t_3\}=\{u', u'+1, t'\}
\]
as multisets. Up to reordering of the indices, we may assume 
\[
t_1=u', t_2=u'+1, t_3=t',
\]
which makes \eqref{E:multisets_GL(3)} into
\begin{equation}\label{E:multisets_GL(3)-1}
\{t_1-\gamma_1, t_1+1-\gamma_2, t_3-\gamma_3\}
=\{t_1-\delta, t_1+1-\delta, t_3-\gamma'\}.
\end{equation}
Also note that the central character assumption \eqref{E:central} implies
\[
-(\ep_1+\ep_2+\ep_3)+u'+u'+1+t'=2u'-N'-\ep'+t',
\]
which gives
\[
N'=(\ep_1+\ep_2+\ep_3)-\ep'-1.
\]
Since $N'>0$, we must always have
\begin{equation}\label{E:epsilon}
\ep_1+\ep_2+\ep_3-\ep'>1.
\end{equation}

By letting $\delta=1$ in \eqref{E:multisets_GL(3)-1}, we have
\begin{equation}\label{E:multisets_GL(3)-2}
\{t_1-\gamma_1, t_1+1-\gamma_2, t_3-\gamma_3\}
=\{t_1-1, t_1, t_3-\gamma'\}.
\end{equation}
Assume $t_1-\gamma_1$ is a minimal element in the left hand
side . Then we must have $\gamma_1=2$ (so $\ep_1=1$), because $t_1-1$
is in the left hand side and so we must have
$t_1-\gamma_1\preceq t_1-1$, and hence we have
\[
\{t_1-2, t_1+1-\gamma_2, t_3-\gamma_3\}
=\{t_1-1, t_1, t_3-\gamma'\}.
\]
Then we must have $t_1-2=t_3-\gamma'$, which implies
\[
\{t_1+1-\gamma_2, t_1-2-\gamma_3+\gamma'\}
=\{t_1-1, t_1\}.
\]
Considering $1-\gamma_2$ is always odd, we conclude
$t_1+1-\gamma_2=t_1-1$, which implies $\gamma_2=2$ (so $\ep_2=1$), and
$t_1-2-\gamma_3+\gamma'=t_1$, which implies $\gamma_3=0$ (so
$\ep_3=0$) and $\gamma'=2$ (so $\ep'=1$). But this contradicts the
condition \eqref{E:epsilon} coming from the central character
assumption. Hence
$t_1-\gamma_1$ cannot be minimal in the left hand side of
\eqref{E:multisets_GL(3)-2}. 

Suppose now that $t_1+1-\gamma_2$ is minimal.  Then this implies that $\gamma_1=0$
and $\gamma_2=2$, in which case $\ep_1=0$ and
$\ep_2=1$. Then \eqref{E:multisets_GL(3)-2} 
implies $t_3-\gamma_3=t_3-\gamma'$, which implies $\gamma_3=\gamma'$, which in turn implies 
$\ep_3=\ep'$. Then the left hand side of \eqref{E:epsilon}
will be $1$, which is a contradiction. Hence neither $t_1-\gamma_1$
nor $t_2+1-\gamma_2$ can be minimal.

So $t_3-\gamma_3$ can be the only minimal
element, which means
\[
t_3-\gamma_3\prec t_1-1\prec t_1
\] or
\[
t_3-\gamma_3\prec t_1-2\prec t_1+1-\gamma_2.
\]
The first case comes from the situation where $\gamma_1 = 0,
\gamma_2 = 2$, which we argued earlier cannot happen.  We consider the
second case. Notice that all the elements are comparable and the inequalities
are all strict. Hence the same has to happen in the right hand
side. But one can easily see that this is impossible because the right
hand side already contains $t_1-1$ and $t_1$ while the second largest
element in the left hand side is $t_1-2$. 

Hence we conclude that we cannot have both \eqref{E:1+1+1} and
\eqref{E:2+1} at the same time; namely, if $\varphi$ is a sum of three irreducible
representations, so is $\varphi'$, and if $\varphi$ is a sum of two
irreducible representations, so is $\varphi'$. 

First, assume
\begin{align*}
\varphi&=\la_{\ep_1, t_1}\oplus\la_{\ep_2, t_2}\oplus\la_{\ep_3, t_3}\\
\varphi'&=\la_{\ep'_1, t'_1}\oplus\la_{\ep'_2, t'_2}\oplus\la_{\ep'_3, t'_3}.
\end{align*}
Then the argument is similar to Step 2 of the real case. Namely by
twisting by $\la_{0, s}$ and equating the $L$-factors, one can get
\[
\{t_1,t_2,t_3\}=\{t'_1,t'_2,t'_3\}
\]
as multisets. Also by twisting by $\la_{1,s}$ and equating the $L$-factors, one can get
\[
\{t_1-\gamma_1,t_2-\gamma_2,t_3-\gamma_3\}=
\{t'_1-\gamma'_1,t'_2-\gamma'_2,t'_3-\gamma'_3\}
\]
as multisets, where $\gamma_i=2$ if $\ep_i=1$, and $\gamma_i=0$ if
$\ep_i=0$, and similarly for $\gamma'_i$. Then the argument of Step 2 of the real case goes
through with the $\ep$'s replaced by $\gamma$'s. Hence we get $\varphi=\varphi'$.

Next, assume
\begin{align*}
\varphi&=\varphi_{-N, t}\oplus\la_{\ep, u}\\
\varphi'&=\varphi_{-N', t'}\oplus\la_{\ep', u'}.
\end{align*}
By twisting $\la_{\delta, s}$, equating the $L$-factors and using the
duplication formula, one obtains
\[
\{t-\delta, t+1-\delta, u-\gamma\}=
\{t'-\delta, t'+1-\delta, u'-\gamma'\}
\]
where $\gamma=2$ if $\ep=\delta=1$ and 0 otherwise, and similarly for
$\gamma'$. By examining minimal elements, one can see that this implies
$t=t'$, $u=u'$ and $\gamma=\gamma'$ (so $\ep=\ep'$). (The detail is left to the reader.) Finally the
central character assumption implies $N=N'$. The proposition follows.
\end{proof}

Let us remark that even when $n=2$, the assumption on the central
character is necessary. Consider the following example.
\begin{Ex}
Let 
\begin{align*}
\varphi&=\varphi_{-N, t}\\
\varphi'&=\varphi_{-N', t}
\end{align*}
with $N>0, N'>0$ but $N\neq N'$. Then for any
$\la_{\delta, s}$, we have
\[
L(\varphi\otimes\la_{\delta, s})=L(\varphi'\otimes\la_{\delta,s})
=2(2\pi)^{-(t+s-\delta)}\Gamma(t+s-\delta).
\]
Of course, there are infinitely many such choices for $N$ and
$N'$. If we further assume $N=N'\pmod{4}$, not only the $L$-factors
but the $\epsilon$-factors coincide. Indeed, we can compute
\[
\epsilon(\varphi\otimes\la_{\delta, s},
  \psi_\R)=\epsilon(\varphi'\otimes\la_{\delta, s}, \psi_\R)
=-i^{N+1}=-i^{N'+1},
\]
and again there are infinitely many such choices.
\end{Ex}

\subsection{Case for $\GL_4(\R)$}
Finally, the following example for $\GL_4(\R)$ shows that,
even with the central character assumption and the equality of the
$\epsilon$-factors, one will still need $\GL(2)$-twists.
\begin{Ex}
Let
\begin{align*}
\varphi&=\varphi_{-N_1, t_1}\oplus\varphi_{-N_2, t_2}\\
\varphi'&=\varphi_{-N'_1, t'_1}\oplus\varphi_{-N'_2, t'_2}.
\end{align*}
Let $\omega$ and $\omega'$ be the central characters of the
representations corresponding to $\varphi$ and $\varphi'$,
respectively. Then by \eqref{E:central_character} we know
\begin{align*}
\omega(r)&=r^{-(N_1+N_2)+2(t_1+t_2)}\\
\omega'(r)&=r^{-(N'_1+N'_2)+2(t'_1+t'_2)}
\end{align*}
for $r\in\R^\times$. For each $\la_{\delta, s}$, we have
\begin{align*}
L(\varphi\otimes\la_{\delta,s})&=4(2\pi)^{-(2s+t_1+t_2-2\delta)}
\Gamma(s+t_1-\delta)\Gamma(s+t_2-\delta)\\
L(\varphi'\otimes\la_{\delta,s})&=4(2\pi)^{-(2s+t'_1+t'_2-2\delta)}
\Gamma(s+t'_1-\delta)\Gamma(s+t'_2-\delta)
\end{align*}
and
\begin{align*}
\epsilon(\varphi\otimes\la_{\delta,s}, \psi_\R)&=i^{N_1+N_2+2}\\
\epsilon(\varphi'\otimes\la_{\delta,s}, \psi_\R)&=i^{N'_1+N'_2+2}.
\end{align*}
Hence if we have
\[
N_1+N_2=N'_1+N'_2\quad\text{and}\quad \{t_1, t_2\}=\{t'_1, t'_2\}, 
\]
then the $L$-factors, $\epsilon$-factors and central characters will
all coincide. Of course there are infinitely many such choices.
\end{Ex}

\end{document}